\documentclass[12pt]{amsart}

\usepackage{amsfonts, amsthm, amsmath}

\usepackage{rotating}

\usepackage{tikz}

\usepackage{amscd}

\usepackage[latin2]{inputenc}

\usepackage{t1enc}

\usepackage[mathscr]{eucal}

\usepackage{indentfirst}

\usepackage{graphicx}

\usepackage{graphics}

\usepackage{pict2e}

\usepackage{mathrsfs}

\usepackage{enumerate}
\usepackage[pagebackref]{hyperref}
\hypersetup{backref, pagebackref, colorlinks=true}
\usepackage{cite}
\usepackage{color}
\usepackage{epic}
\usepackage{hyperref} 
\numberwithin{equation}{section}
\theoremstyle{plain}

\newtheorem{theorem}{Theorem}[section]

\theoremstyle{definition}

\newtheorem{example}[theorem]{Example}

\newtheorem{?}[theorem]{Problem}

\def\boxit#1{\leavevmode\hbox{\vrule\vtop{\vbox{\kern.33333pt\hrule
    \kern1pt\hbox{\kern1pt\vbox{#1}\kern1pt}}\kern1pt\hrule}\vrule}}

\usepackage{collectbox}



\begin{document}
\title[Generalizing a partition theorem of Andrews]{Generalizing a partition theorem of Andrews}

\author[S. Fu]{Shishuo Fu}
\address[Shishuo Fu]{College of Mathematics and Statistics, Chongqing University, Huxi Campus LD506, Chongqing 401331, P.R. China}
\email{fsshuo@cqu.edu.cn}

\author[D. Tang]{Dazhao Tang}

\address[Dazhao Tang]{College of Mathematics and Statistics, Chongqing University, Huxi Campus LD206, Chongqing 401331, P.R. China}
\email{dazhaotang@sina.com}

\date{\today}

\begin{abstract}
Abstract: Motivated by Andrews' recent work related to Euler's partition theorem, we consider the set of partitions of an integer $n$ where the set of even parts has exactly $j$ elements, versus the set of partitions of $n$ where the set of repeated parts has exactly $j$ elements. These two sets of partitions turn out to be equinumerous, and this naturally encloses Euler's theorem and Andrews' theorem as two special cases. We give two proofs, one using generating function, and the other is a direct bijection that builds on Glaisher's bijection
\end{abstract}



\maketitle
\section{Introduction}
A \emph{partition} \cite{Andr1} $\pi$ of a positive integer $n$ is a finite sequence of weakly decreasing positive integers $\pi_{1}\geq\pi_{2}\geq\cdots\geq\pi_{r}>0$ such that $\sum_{i=1}^{r}\pi_{i}=n$. The $\pi_{i}$'s are called the \emph{parts} of $\pi$. Sometimes it is useful to use a notation that makes explicit the number of times that a particular integer occurs as a part and we write
\begin{align*}
\pi=(1^{m_{1}}2^{m_{2}}3^{m_{3}}\cdots),
\end{align*}
where exactly $m_{j}$ of the $\pi_{i}$'s are equal to $j$.

In 1748, Euler \cite{Eul} gave the generating function proof of the following partition theorem:
\begin{theorem}\label{Euler}
The number of partitions of $n$ into odd parts equals the number of partitions of $n$ into distinct parts.
\end{theorem}
When odd numbers are viewed as numbers not divisible by $2$, the following theorem of Glaisher \cite[Corollary~1.3]{Andr1} can be viewed as a natural generalization of Theorem~\ref{Euler}.
\begin{theorem}\label{Glaisher}
The number of partitions of $n$ into parts not divisible by $d$ equals the number of partitions of $n$ into parts each occuring no more than $d-1$ times.
\end{theorem}
Partitions into parts not divisible by $d$ are usually called \emph{$d$-regular partitions}, and this is the term we use in the sequel.

In order to solve two conjectures of Beck posted on the OEIS (see \cite{OEIS1} and \cite{OEIS2}), Andrews \cite{Andr2} established the following
\begin{theorem}[Theorem~1 in \cite{Andr2}]\label{Andrews}
If $a(n)$ denotes the number of partitions of $n$ such that the set of even parts has only one element (possibly repeated), $b(n)$ counts the difference between the number of parts in the odd partitions of $n$ and the number of parts in the distinct partitions of $n$, and $c(n)$ denotes the number of partitions of $n$ where exactly one part is repeated. Then $a(n)=b(n)=c(n)$ for all $n\geq1$.
\end{theorem}
In particular, the assertion $a(n)=b(n)$ verifies Beck's conjecture listed in \cite{OEIS1}.
In this note, we give two generalizations of Andrews' theorem. Let $O_{j,k}(n)$ denote the number of partitions of $n$ such that there are exactly $j$ different parts $\equiv0\pmod{k}$ (possibly repeated), and $D_{j,k}(n)$ denote the number of partitions of $n$ such that there are exactly $j$ different parts repeated at least $k$ times. We obtain the first generalization of Theorem~\ref{Andrews}.
\begin{theorem}\label{jk analog}
For all $j, n\geq0$ and $k\geq1$, $O_{j,k}(n)=D_{j,k}(n)$.
\end{theorem}
Evidently, $j=0,k=2$ gives us Theorem~\ref{Euler}, $j=0, k\geq 2$ produces Theorem~\ref{Glaisher}, and $j=1, k=2$ leads us back to the $a(n)=c(n)$ part of Theorem~\ref{Andrews}.

Next, in order to state another generalization of Theorem~\ref{Andrews} involving $b(n)$, let $E_{k}(n)$ denote the excess of the number of parts $\equiv1\pmod{k}$ in the $k$-regular partitions of $n$, over the number of distinct parts in the partitions of $n$ with no parts being repeated more than $k-1$ times, then we have the following theorem, which is precisely Theorem~\ref{Andrews} when we set $k=2$.
\begin{theorem}\label{k analog:a=b=c}
For all $n\geq0$ and $k\geq2$, $O_{1,k}(n)=D_{1,k}(n)=E_{k}(n)$.
\end{theorem}

\section{Proof of Theorems}
In this section, we first give a generating function proof as well as a bijective proof of Theorem~\ref{jk analog}. This partially answers to Andrews' request \cite{Andr2}. Then we follow Andrews' method and give a generating function proof of Theorem~\ref{k analog:a=b=c}.

\begin{proof}[Generating function proof of Theorem \ref{jk analog}]\ \\
Firstly, define the following two bivariate generating functions
\begin{align*}
\mathcal{O}_{k}(z,q) &:=\sum_{n=0}^{\infty}\sum_{j=0}^{\infty}O_{j,k}(n)z^{j}q^{n},\\
\mathcal{D}_{k}(z,q) &:=\sum_{n=0}^{\infty}\sum_{j=0}^{\infty}D_{j,k}(n)z^{j}q^{n}.
\end{align*}

Obviously, we have
\begin{align}
\mathcal{O}_{k}(z,q) &=\prod_{m=1}^{\infty}\left(1+zq^{km}+zq^{2km}+zq^{3km}+\cdots\right)\prod_{n=1\atop n\not\equiv0\pmod{k}}^{\infty}\dfrac{1}{1-q^{n}}\label{gene fun1}\\
 &=\prod_{m=1}^{\infty}\left(1+\dfrac{zq^{km}}{1-q^{km}}\right)\prod_{n=1}^{\infty}\dfrac{1-q^{kn}}{1-q^{n}},\nonumber
\end{align}
and
\begin{align}
\mathcal{D}_{k}(z,q) &=\prod_{n=1}^{\infty}\left(1+q^{n}+q^{2n}+\cdots+q^{(k-1)n}+zq^{kn}+zq^{(k+1)n}+\cdots\right)\label{gene fun2}\\
 &=\prod_{n=1}^{\infty}\left(1+q^{n}+q^{2n}+\cdots+q^{(k-1)n}\right)\left(1+zq^{kn}+zq^{2kn}+\cdots\right)\nonumber\\
 &=\prod_{n=1}^{\infty}\dfrac{1-q^{kn}}{1-q^{n}}\prod_{m=1}^{\infty}\left(1+\dfrac{zq^{km}}{1-q^{km}}\right).\nonumber
\end{align}
Equating the coefficients of $z^{j}q^{n}$ for $\mathcal{O}_{k}(z,q)$ and $\mathcal{D}_{k}(z,q)$ finishes the proof.
\end{proof}
\begin{proof}[Bijective proof of Theorem \ref{jk analog}]\ \\
We construct a family of bijections $\{\varphi_{k}\}_{k\geq 1}$ on $\mathcal{P}$, the set of all partitions. For $k=1$, we simply take $\varphi_1$ to be the identity map since in this case the two sets of partitions in question coincide with each other. For general $k\geq 2$, let $\pi=\pi_{1}+\cdots+\pi_{r}$ be a partition. We discuss by two cases, depending on whether the part $\pi_{i}$ is divisible by $k$ or not.
\begin{itemize}
\item $\pi_{i}\equiv0\pmod{k}$: Denote by $\overline{\pi}$ the partition composed of all parts $\equiv0\pmod{k}$ in $\pi$. In this case, we simply map the part $\pi_{i}$ to the $k$-copies of equal parts
\begin{align*}
\underbrace{\dfrac{\pi_{i}}{k}+\dfrac{\pi_{i}}{k}+\cdots+\dfrac{\pi_{i}}{k}}_{k~\textrm{copies}}.
\end{align*}
Collect all parts converted this way and denote as $\overline{\lambda}$ the partition composed of them. Consequently, the parts in $\overline{\lambda}$ are all repeated at least $k$ times, which means we can identify them when we apply the inverse map. Actually the multiplicity of every part in $\overline{\lambda}$ must be a multiple of $k$.

\item $\pi_{i}\not\equiv0\pmod{k}$: Denote by $\tilde{\pi}$ the partition composed of all parts $\not\equiv0\pmod{k}$ in $\pi$. Assume $\widetilde{\pi}$ has the form
\begin{align*}
\underbrace{q_{1}+q_{1}+\cdots+q_{1}}_{m_{1}~\textrm{times}}+\underbrace{q_{2}+q_{2}+\cdots+q_{2}}_{m_{2}~\textrm{times}}+\cdots+\underbrace{q_{s}+q_{s}+\cdots+q_{s}}_
{m_{s}~\textrm{times}}.
\end{align*}
We now consider two cases again, depending on whether $m_{i}$ is no less than $k$ or not. If $m_{i}\geq k$, then we split $k$ copies off of the part $q_{i}$ and combine (add) these $k$ parts of equal size into a new part, and we keep splitting $k$ copies off of the remaining parts (if any) and combining these parts of equal size into a new parts until there are no more parts repeated more than $k-1$ times, then we eventually obtain a partition $\widetilde{\lambda}$ where all parts occur less than $k$ times. This is essentially Glaisher's map for proving Theorem~\ref{Glaisher}.

Finally let $\varphi_k(\pi):=\lambda=\overline{\lambda}\bigcup\widetilde{\lambda}$, where the union $\overline{\lambda}\bigcup\widetilde{\lambda}$ is the partition consists of all parts in $\overline{\lambda}$ and $\widetilde{\lambda}$ (arranged in nonincreasing order). We observe that the number of different parts $\equiv0\pmod{k}$ in partition $\pi$ is the same as the number of parts that are repeated more than $k-1$ times in $\lambda$ via the above map $\varphi_k$.
\end{itemize}

The inverse map should mostly be clear, with the only caveat being that when we decompose $\lambda$ as $\overline{\lambda}\bigcup\widetilde{\lambda}$, $\overline{\lambda}$ consists of those parts in $\lambda$ being repeated, say $m\geq k$ times, but in general not all $m$ copies of this part is included in $\overline{\lambda}$. More precisely, we can uniquely write $m=tk+l$, with $0\leq l<k$, then exactly $tk$ copies of this part are included in $\overline{\lambda}$, with the remaining $l$ copies included in $\widetilde{\lambda}$.
\end{proof}
\begin{example}\label{exam2}
Table \ref{Tab2} lists all partitions of $16$ enumerated by $O_{2,4}(16)$ and $D_{2,4}(16)$ respectively, with one-to-one correspondence between each row via $\varphi_{4}$.
\end{example}
\begin{table}[tbp]\caption{One-to-one correspondence in Example \ref{exam2}}\label{Tab2}
\centering
\begin{tabular}{|c||c|}
$\equiv0\pmod{4}$ &$f_{i}\geq4$\\
$(12,4)$ &$(3,3,3,3,1,1,1,1)$\\
$(8,4,4)$ &$(2,2,2,2,1,1,1,1,1,1,1,1)$\\
$(8,4,3,1)$ &$(3,2,2,2,2,1,1,1,1,1)$\\
$(8,4,2,2)$ &$(2,2,2,2,2,2,1,1,1,1)$\\
$(8,4,2,1,1)$ &$(2,2,2,2,2,1,1,1,1,1,1)$\\
$(8,4,1,1,1,1)$ &$(4,2,2,2,2,1,1,1,1)$
\end{tabular}
\end{table}

\begin{proof}[Proof of Theorem~\ref{k analog:a=b=c}]
The proof of Theorem \ref{k analog:a=b=c} relies on generating functions and differentiation. According to \eqref{gene fun1} and \eqref{gene fun2}, by extracting the coefficients of $z$ in both generating functions we have
\begin{align}\label{gf:o1k}
\sum_{n=0}^{\infty}O_{1,k}(n)q^{n}=\sum_{n=0}^{\infty}D_{1,k}(n)q^{n}=\sum_{m=1}^{\infty}\dfrac{q^{km}}{1-q^{km}}\prod_{n=1}^{\infty}\dfrac{1-q^{kn}}{1-q^{n}}.
\end{align}

Moreover, in the infinite product
\begin{align*}
\prod_{n=1}^{\infty}\dfrac{1}{(1-zq^{k(n-1)+1})(1-q^{k(n-1)+2})\cdots(1-q^{k(n-1)+k-1})},
\end{align*}
the coefficient of $z^{M}q^{N}$ is the number of $k$-regular partitions of $N$ with exactly $M$ parts $\equiv1\pmod{k}$, while in the infinite product
\begin{align*}
\prod_{n=1}^{\infty}\left(1+zq^{n}+zq^{2n}+\cdots+zq^{(k-1)n}\right),
\end{align*}
the coefficient of $z^{M}q^{N}$ is the number of partitions of $N$ into $M$ different parts, each part repeated less than $k$ times.

Thus if we differentiate each of these functions with respect to $z$ we will then be counting each partition with $M$ designated parts with weight $M$. Consequently
\begin{align*}
\sum_{n=0}^{\infty}D_k(n)q^{n}
 =&\dfrac{\partial}{\partial z}\bigg|_{z=1}\bigg(\prod_{n=1}^{\infty}\dfrac{1}{(1-zq^{k(n-1)+1})(1-q^{k(n-1)+2})\cdots(1-q^{k(n-1)+k-1})} \\
 & -\prod_{n=1}^{\infty}\big(1+zq^{n}+zq^{2n}+\cdots+zq^{(k-1)n}\big)\bigg)\\
 =&\sum_{n=1}^{\infty}\dfrac{q^{k(n-1)+1}}{(1-q^{k(n-1)+1})^{2}(1-q^{k(n-1)+2})\cdots(1-q^{k(n-1)+k-1})}\times\\
 &\prod_{m=1\atop m\neq n}^{\infty}\dfrac{1}{(1-q^{k(m-1)+1})(1-q^{k(m-1)+2})\cdots(1-q^{k(m-1)+k-1})}\\
 &-\sum_{n=1}^{\infty}\bigg(q^{n}+q^{2n}+\cdots+q^{(k-1)n}\bigg)\prod_{m=1\atop m\neq n}^{\infty}\big(1+q^{m}+q^{2m}+\cdots+q^{(k-1)m}\big)\\
 =&\prod_{m=1}^{\infty}\dfrac{1-q^{km}}{1-q^{m}}\sum_{n=1}^{\infty}\dfrac{q^{k(n-1)+1}}{1-q^{k(n-1)+1}}-\sum_{n=1}^{\infty}\dfrac{q^{n}-q^{kn}}{1-q^{kn}}\prod_{m=1}^{\infty}\dfrac{1-q^{km}}{1-q^{m}}\\
 =&\sum_{n=1}^{\infty}\dfrac{q^{kn}}{1-q^{kn}}\prod_{m=1}^{\infty}\dfrac{1-q^{km}}{1-q^{m}},
\end{align*}
where the last equality follows from
\begin{align*}
\sum_{n=1}^{\infty}\dfrac{q^{k(n-1)+1}}{1-q^{k(n-1)+1}}=\sum_{n=1}^{\infty}\sum_{m=1}^{\infty}q^{m(k(n-1)+1)}=\sum_{m=1}^{\infty}\dfrac{q^{m}}{1-q^{km}}.
\end{align*}

Comparing the final expression with \eqref{gf:o1k} completes the proof.
\end{proof}
\begin{example}\label{exam1}
Table \ref{Tab1} lists all four types of partitions of 8, i.e., partitions with exactly one part divisible by 3, partitions with exactly one part repeated at least 3 times, $3$-regular partitions and partitions all of whose parts are repeated at most 2 times. One checks that $O_{1,3}(8)=D_{1,3}(8)=9, E_{3}(8)=37-28=9$.
\end{example}
\begin{table}[tbp]\caption{Three-way correspondence in Example \ref{exam1}}\label{Tab1}
\centering
\begin{tabular}{c||c||cc}
$\equiv0\pmod{3}$ &$f_{i}\geq3$ &3-regular partitions &$f_{i}\leq2$\\
$(6,2)$ &$(2,2,2,2)$ &$(8)$ &$(8)$\\
$(6,1,1)$ &$(2,2,2,1,1)$ &$(7,1)$ &$(7,1)$\\
$(3,3,2)$ &$(2,1,1,1,1,1,1)$ &$(2,2,2,2)$ &$(6,2)$\\
$(3,3,1,1)$ &$(1,1,1,1,1,1,1,1)$ &$(2,2,2,1,1)$ &$(6,1,1)$\\
$(5,3)$ &$(5,1,1,1)$ &$(5,1,1,1)$ &$(5,3)$\\
$(4,3,1)$ &$(4,1,1,1,1)$ &$(5,2,1)$ &$(5,2,1)$\\
$(3,2,2,1)$ &$(2,2,1,1,1,1)$ &$(4,4)$ &$(4,4)$\\
$(3,2,1,1,1)$ &$(3,2,1,1,1)$ &$(4,1,1,1,1)$ &$(4,3,1)$\\
$(3,1,1,1,1,1)$ &$(3,1,1,1,1,1)$ &$(4,2,2)$ &$(4,2,2)$\\
 & &$(4,2,1,1)$ &$(4,2,1,1)$\\
 & &$(2,1,1,1,1,1,1)$ &$(3,3,2)$\\
 & &$(1,1,1,1,1,1,1,1)$ &$(3,3,1,1)$\\
 & &$(2,2,1,1,1,1)$ &$(3,2,2,1)$
\end{tabular}
\end{table}

\section{Conclusion}
Several questions present themselves from this study. We collect them here to motivate further work.
\begin{enumerate}[1)]
\item The bijection between partition sets enumerated by $a(n)$ and $b(n)$, or between $c(n)$ and $b(n)$ is still unclear.
\item A more ambicious task than 1) is to ``complete the puzzle'' and find the family of partitions enumerated by certain $E_{j,k}(n)$ such that $O_{j,k}(n)=D_{j,k}(n)=E_{j,k}(n)$ for all $j\geq 1$ and our Theorem~\ref{k analog:a=b=c} becomes the special case of $j=1$. The differentiation technique suggests the complexity of the computation and reverse engineering our proof to track down this third family becomes less optimistic.
\item Theorem~\ref{jk analog} sets the classical Euler's partition theorem at the center of a natural family of partition theorems. On a general note, can we try this new perspective with other classical partition theorems? More precisely, for a partition theorem that links those partitions of $n$ without certain parts of ``type A'' (even parts in the case of Euler's theorem), with those partitions of $n$ without certain parts of ``type B'' (repeated parts in the case of Euler's theorem), we can consider what happens if we require that there are exactly $j$ different parts of type A (resp. type B) in the first (resp. the second) set of partitions.
\end{enumerate}

\section*{Acknowledgement}
Both authors were supported by the Fundamental Research Funds for the Central Universities (No.~CQDXWL-2014-Z004) and the National Science Foundation of China (No.~115010\\61).


\begin{thebibliography}{99}

\bibitem{Andr1}G.~E.~Andrews, \emph{The Theory of Partitions}, Encyclopedia of Mathematics and Its Applications, Vol. 2 (G.-C. Rota, ed.), Addison-Wesley, Reading, 1976 (Reprinted: Cambridge Univ. Press, London and New York, 1984).

\bibitem{Andr2}G.~E.~Andrews, \emph{Euler's partition identity and two problems of George Berk}. Preprint.

\bibitem{Eul}L.~Euler, \emph{Introductio in Analysin Infinitorum}, vol. 2. MM Bousquet, Hanson (1748).

\bibitem{OEIS1}The On-Line Encyclopedia of Integer Sequences, \href{http://oeis.org/A090867}{oeis: A090867}.

\bibitem{OEIS2}The On-Line Encyclopedia of Integer Sequences, \href{http://oeis.org/A265251}{oeis: A265251}.


\end{thebibliography}
\end{document}